\documentclass[reqno,12pt]{amsart}
\newcommand{\be}{\begin{eqnarray}}
\newcommand{\ee}{\end{eqnarray}}
\newcommand{\ben}{\begin{eqnarray*}}
\newcommand{\een}{\end{eqnarray*}}

\newtheorem{theorem}{Theorem}

\newtheorem{definition}[theorem]{Definition}

\newtheorem{lemma}[theorem]{Lemma}

\newtheorem{proposition}[theorem]{Proposition}
\newtheorem{remark}[theorem]{Remark}

{\catcode`\@=11\global\let\AddToReset=\@addtoreset
\AddToReset{equation}{section}

\AddToReset{theorem}{section}

\usepackage{graphicx}
\usepackage{psfrag}
\usepackage{epsfig,color}

\newcommand{\RR}{\mathbb R}

\begin{document}
\title[Existence of sign changing solutions]
      {On the existence of sign changing bound state solutions of
            a quasilinear equation   }\thanks{This research was supported by
        FONDECYT-1110074 for the first author, and
        FONDECYT-1110268  and FONDECYT-1110003 for the second and third author.}
\author{Carmen Cort\'azar}
\address{Departamento de Matem\'atica, Pontificia
        Universidad Cat\'olica de Chile,
        Casilla 306, Correo 22,
        Santiago, Chile.}
\email{\tt ccortaza@mat.puc.cl}
\author{Marta Garc\'{\i}a-Huidobro}
\address{Departamento de Matem\'atica, Pontificia
        Universidad Cat\'olica de Chile,
        Casilla 306, Correo 22,
        Santiago, Chile.}
\email{\tt mgarcia@mat.puc.cl}
\author{Cecilia S. Yarur}
\address{Departamento de Matem\'atica y C.C.,
        Universidad de Santiago de Chile,
       Casilla 307, Correo 2, Santiago, Chile}
\email{\tt cyarur@usach.cl}



\begin{abstract}
In this paper we  establish the existence of  bound state solutions having a prescribed number of sign change for
$$ \Delta_m u
+f(u)=0,\quad x\in \RR^N, N\ge m>1, \leqno( P) $$
where $\Delta_m u=\nabla\cdot(|\nabla u|^{m-2}\nabla u)$.
Our result is new even for the case of the Laplacian ($m=2$).
\end{abstract}

\maketitle

\section{Introduction and main results}

In this paper we  establish the existence of higher bound state solutions to
$$ \Delta_m u
+f(u)=0,\quad x\in \RR^N, N\ge m>1, \leqno( P) $$
where $\Delta_m u=\nabla\cdot(|\nabla u|^{m-2}\nabla u)$. To this end we consider the radial version of $(P)$, that is
\begin{eqnarray}\label{eq2}
\begin{gathered}
(\phi_m(u'))'+\frac{N-1}{r}\phi_m(u')+f(u)=0,\quad r>0,\quad N\ge m>1,\\
u'(0)=0,\quad \lim\limits_{r\to\infty}u(r)=0,
\end{gathered}
\end{eqnarray}
where $\phi_m(x)=|x|^{m-2}x$, $x\not=0$, and $\phi_m(0)=0$.

Any nonconstant solution to \eqref{eq2} is called a bound state solution. Bound state solutions such that $u(r)>0$ for all $r>0$, are referred to as a first bound state solution, or  a ground state solution.

The existence of a first bound state for $(P)$ has been established by many authors under different regularity and growth assumptions on the nonlinearity $f$, both for the Laplacian operator and the degenerate Laplacian operator, see for example \cite{ap1, ap2}, \cite{b-l1} and \cite{fr}  in the case of a regular $f$ ($f\in C[0,\infty)$) for the case of the semilinear equation, and \cite{fls}, \cite{gst}  and \cite{fg} for both the singular and  regular case in the quasilinear situation.   In the case of the Laplacian, Berestycki and Lions in \cite{b-l2}, proved the existence of infinitely many radially symmetric bound state solutions by using  variational methods when $f$ is an odd function satisfying very mild assumptions, but the existence of solutions with prescribed number of zeros was as an open question for many years.  Jones and K\"upper in \cite{jk} gave a positive answer to this question using a dynamical systems approach and the Conley index for an $f$ which at infinity grows like $|u|^{\sigma}u$, with $\sigma+1<(N+2)/(N-2)$ and $f$ is not necessarily odd.
In \cite{mtw}, for the Laplacian operator and $f$ satisfying a similar superlinear type condition at infinity, the authors prove the existence of bound states of any order by means of a shooting method and a scaling argument.
Grillakis, in \cite{grillakis}, studied the same problem for the Laplacian operator using the degree of a map from a Banach space to itself, allowing less smoothness in the nonlinear term, imposing a  different subcritical restriction at infinity, and having a more restrictive growth condition for $f$ at $0$, namely, that $f$ is Lipschitz at $0$. Finally we mention the work in \cite{bdo2003}, where the authors treat the same problem for the Laplacian operator with the specific  nonlinearity $u-|u|^{-a}u$, with $a\in(0,1)$.

\medskip

We will assume that the function  $f$ satisfies
\begin{enumerate}
\item[$(f_1)$] $f\in C(\mathbb R)$.
\medskip

\item[$(f_2)$]
\begin{enumerate}
\item[(i)] There exist $\beta^-<0<\beta^+$ such that if we set $F(s)=\int_0^sf(t)dt$, it holds that $F(s)<0$ for all$s\in(\beta^-,\beta^+)\setminus\{0\}$, $F(\beta^{\pm})=0$ and $f(s)>0$ for all $s\in(\beta^+,\gamma^+)$ and $f(s)<0$ for all $s\in(\gamma^-,\beta^-)$, where
\ben
\gamma^+&=&\min\{s>\beta^+\ |\ f(s)=0\},\quad\mbox{$\gamma^+=\infty$ if $f(s)>0$ for all $s>\beta^+$,}\\
\gamma^-&=&\max\{s<\beta^-\ |\ f(s)=0\},    \quad\mbox{ $\gamma^-=-\infty$ if $f(s)<0$ for all $s<\beta^-$},
\een

\item[(ii)] $\lim\limits_{s\to\gamma^-}F(s)=L=\lim\limits_{s\to\gamma^+}F(s)$ ($L$ may be finite or not).

\item[(iii)] $f$ is locally Lipschitz in $(\gamma^-,0)\cup(0,\gamma^+)$.
\end{enumerate}
\medskip

\item[$(f_3)$] Let $Q(s):=mNF(s)-(N-m)sf(s)$.

\noindent (a) If $\gamma^+=\infty$ we assume
\begin{enumerate}
\item[(i)] There exists $\bar\beta>\beta^+$   such that  $Q(s)\ge 0$ for all $s\ge\bar\beta$,
\item[(ii)] There exists $\theta\in(0,1)$ such that
\be\label{cond}\lim_{s\to\infty}\Bigl(\inf_{s_1,s_2\in[\theta s,s]}Q(s_2)\Bigl(\frac{|s|^{m-2}s}{f(s_1)}\Bigr)^{N/m}\Bigr)=\infty\ee
\end{enumerate}
(b) If $\gamma^-=-\infty$ we assume
\begin{enumerate}
\item[(i)] There exists $-\bar\beta<\beta^-$   such that  $Q(s)\ge 0$ for all $s\le-\bar\beta$,
\item[(ii)] There exists $\theta\in(0,1)$ such that
\be\label{cond-}\lim_{s\to-\infty}\Bigl(\inf_{s_1,s_2\in[s,\theta s]}Q(s_2)\Bigl(\frac{|s|^{m-2}s}{f(s_1)}\Bigr)^{N/m}\Bigr)=\infty\ee
\end{enumerate}
\medskip

\item[$(f_4)$] \begin{enumerate}
\item [(a)]If $\gamma^+$ is finite, we assume there exists $L_0>0$ and $\bar\beta>\beta^+$ such that
\be\label{condgamma+}
f(s)\le L_0(\gamma^+-s)^{m-1}\quad\mbox{for all $s\in[\bar\beta,\gamma^+)$,}
\ee
\item[(b)]If $\gamma^-$ is finite, we assume there exists $L_0>0$ and $\bar\beta>-\beta^-$ such that
\be\label{condgamma-}
-f(s)\le L_0(s-\gamma^-)^{m-1}\quad\mbox{for all $s\in(\gamma^-,-\bar\beta]$.}
\ee
\end{enumerate}
\end{enumerate}
We note that
condition \eqref{cond} is stronger than the one used in \cite{gst}, where only the limsup is involved. We note that a similar condition was used first by Castro and Kurepa in \cite{ck} for the Laplace operator and later modified in \cite{ghmsc} and \cite{ghmz}, where it was used to establish the existence of infinitely many  solutions  to the Dirichlet problem  in a ball   for a general quasilinear operator.

The uniqueness of the first bound state solution of \eqref{eq2}  has been exhaustively studied during the last thirty years, see
for example the works \cite{cl},  \cite{coff},
\cite{cfe1}, \cite{cfe2}, \cite{fls}, \cite{Kw}, \cite{m}, \cite{ms},
\cite{pel-ser1}, \cite{pel-ser2},  \cite{pu-ser}, \cite{st}. The uniqueness of higher order bound states was first studied in \cite{troy} for a very special nonlinearity in the semilinear case, and later in \cite{cghy,cghy2} for more general nonlinearities obeying some growth restrictions in the so called  sub-Serrin case. The proofs are made for the semilinear case, but as it is mentioned there, they can  be extended to the quasilinear situation $m>1$.

Our main result  is the following:
\begin{theorem}\label{main}
Assume that $f$ satisfies  $(f_1)$-$(f_4)$. Then given any $k\in\mathbb N$, there exists a solution $u$ of $(P)$ having exactly $k$ zeros in $(0,\infty)$.

\end{theorem}

\begin{remark}\label{rem1}{\rm We will see in section \ref{final} that if there exists $s_0>0$ such that for $|s|\ge s_0$, $sf(s)\ge 0$, $F(s)\ge 0$, and
$$\limsup_{|s|\to\infty}\frac{sf(s)}{F(s)}<m^*,\leqno(SC)$$
where as usual $m^*=\frac{Nm}{N-m}$, then $(f_3)$ is satisfied.

We will also see that $(f_3)$ is satisfied for a function $f$ such that
$$f(s)=\frac{s^{m^*-1}}{(\log s)^\lambda},\quad \mbox{for some $\lambda>m/(N-m)$ and $s$ large,}$$
see \cite[Theorem 2]{fg}. This function {\bf does not} satisfy $(SC)$ above.
We also give an example of a nonlinearity for which $\gamma^+$ and $-\gamma^-$ are finite, and an example  for which $\gamma^+<\infty$ and $\gamma^-=-\infty$, allowing {\bf any growth rate} for $f$ at $\pm\infty$ in the first case and any growth rate for $f$ at $\infty$ in the second case.}
\end{remark}
\medskip

To our knowledge, our result in new even for the case of the Laplacian, that is $m=2$.

\bigskip

In order to prove our result, we will study the behavior of the solutions to the initial value problem
\begin{equation}\label{ivp}
\begin{gathered}
(\phi_m(u'))'+\frac{N-1}{r}\phi_m(u')+f(u)=0,\quad r>0,\quad N\ge m>1,\\
u(0)=\alpha,\quad u'(0)=0,
\end{gathered}\end{equation}
for $\alpha\in(\beta^+,\gamma^+)$. By a solution to \eqref{ivp} we mean a $C^1$ function $u$ such that $\phi_m(u')$ is also $C^1$ in its domain.

The idea is to take $\mathcal N_1$ as the set of initial values $\alpha>\beta^+$ such that a solution $u$ of \eqref{ivp} has at least one simple zero and decompose it as
$\mathcal N_1=\mathcal N_2\cup\mathcal P_2\cup\mathcal G_2$, where $\mathcal N_2$ is the set of initial values $\alpha\in\mathcal N_1$ such that $u$ has at least two simple zeros in $[0,\infty)$,  $\mathcal P_2$ is the set of initial values $\alpha\in\mathcal N_1$ such that $u$ has exactly one simple zero in $[0,\infty)$ and from some point on, $u$ remains negative and bounded above by a negative constant, and $\mathcal G_2$ is the set of $\alpha\in\mathcal N_1$ such that $u$ is a solution to \eqref{eq2} having exactly one sign change in $[0,\infty)$. We prove that $\mathcal N_1$, $\mathcal N_2$, $\mathcal P_2$ are open and nonempty. As $\mathcal N_2$, $\mathcal P_2$ and $\mathcal G_2$ are disjoint, we must have that $\mathcal G_2$ is also nonempty. We repeat this idea to obtain higher order bound states.

In section \ref{prel}, we establish some properties of the solutions to \eqref{ivp}. We restrict its domain  to the set of unique extendibility, and  define some crucial sets of initial values.  Then in section \ref{exist} we prove our main result. In some steps of our proof,  we adapt to our situation  techniques used in \cite{gst}.
Finally, in section \ref{final} we give some examples.

\section{Some properties of the solutions of the initial value problem}\label{prel}

The aim of this section is to establish several properties of the solutions to the initial value problem \eqref{ivp}.

{\color{black}It can be seen, see for example \cite{fls, ns1, ns2}, that solutions are defined and unique at least while they remain nonnegative. Also, if a solution reaches the value zero with a nonzero slope, then this solution can be uniquely continued by considering the equation satisfied by its inverse $r=r(s,\alpha)$ in an appropriate neighborhood of such a point, namely
\begin{eqnarray}\label{r-eq0}
\begin{cases}r'=p\\
p'=\displaystyle\frac{N-1}{m-1}\frac{p^2}{r}+|p|^mpf(s)\\
r(0)=r_0>0,\quad p(0)=p_0\not=0,
\end{cases}
\end{eqnarray}
as the right hand side is Lipschitz in the variable $(r,p)$.

Moreover,
by re-writing the equation in \eqref{ivp} in the form
\be\label{cg1}
\begin{cases}u'=\phi_{m'}(v)\\
v'=\displaystyle-\frac{N-1}{r}v-f(u),\\
u(r_0)=u_0\quad v(r_0)=v_0
\end{cases}
\ee
where $m'=m/(m-1)$, we see that such a solution can be uniquely extended until it reaches a double zero  when $m\le 2$, as in this case the right hand side is Lipschitz in a neighborhood of any point $(u_0,v_0)$ with $u_0\not=0$.
 When $m>2$ this is not clear because $\phi_{m'}$  is not locally Lipschitz near $0$.
 Nevertheless, this problem can be handled if $f(u(r_0))\not=0$}:
Assume for simplicity that $f(u_0)<0$, From the second equation in \eqref{cg1},  if $\delta>0$ is small enough to have that for $|r-r_0|<\delta$,
 $$-\frac{N-1}{r}v-f(u)\ge \frac{1}{2}|f(u_0)|,$$
(which is possible because $v(r_0)=0$), then
 $$v'(r)\ge  \frac{1}{2}|f(u_0)|,$$
implying that
$$v(r)\ge \frac{1}{2}|f(u_0)|(r-r_0).$$
Hence, if $(u_1,v_1)$ and $(u_2,v_2)$ are two solutions of \eqref{cg1}, then for $r>r_0$,  from the mean value theorem, and using that $m'-2<0$, we have
\ben
|(u_1'-u_2')(r)|=|(\phi_{m'}(v_1)-\phi_{m'}(v_2))(r)|&=&(m'-1)|\xi|^{m'-2}|(v_1-v_2)(r)|\\
&\le& C(r-r_0)^{m'-2}|(v_1-v_2)(r)|
\een
for some positive constant $C$ and thus,
$$|(u_1-u_2)(r)|\le C(r-r_0)^{m'-1}||v_1-v_2||,$$
where $||\cdot ||$ represents the usual sup norm in $C[r_0-\delta,r_0+\delta]$.
 Also from the second equation in \eqref{cg1}, using that $f$ is locally Lipschitz we find that
 $$|(v_1-v_2)(r)|\le\int_{r_0}^r|v_1'-v_2'|\le C||v_1-v_2||(r-r_0)+K||u_1-u_2||(r-r_0)$$
for some positive constant $K$. Adding up these two last inequalities we have that
 $$||u_1-u_2||+||v_1-v_2||\le C\delta^{m'-1}||v_1-v_2||+(C+K)\delta||u_1-u_2||,$$
 so choosing $\delta$ small enough we deduce $u_1=u_2$ and $v_1=v_2$ and we have unique extendibility.

\medskip

If $f(u(r_0))=0$, there might be a problem.

Let $u(r)=u(r,\alpha)$ be any solution of \eqref{ivp} such that it reaches a first point $r_0$ where $u'(r_0)=0$ and $f(u(r_0))=0$, and set
\begin{eqnarray}\label{funct-10}
I(r,\alpha)=\frac{|u'(r)|^m}{m'}+F(u(r)).\end{eqnarray}
A simple calculation yields
\begin{eqnarray}\label{I'0}
I'(r,\alpha)=-\frac{(N-1)}{r}|u'(r)|^m,
\end{eqnarray}
and therefore, as $N\ge m>1$, we have that $I$ is decreasing in $r$. As $I(r_0,\alpha)=F(u(r_0))<0$, $u(\cdot)$ cannot change sign again. By our assumptions, $u(r_0)$ is either a local minimum or a local maximum for $F$. If it is a local minimum, then the only possibility is that $u(r)\equiv u(r_0)$ for all $r\ge r_0$. Indeed, for $r\ge r_0$,
$$F(u(r))\le \frac{|u'(r)|^m}{m'}+F(u(r))\le F(u(r_0)),$$
proving the claim. In particular  we have  uniqueness if $f$ has only one positive zero and only one negative zero.

\begin{definition}\label{domain}
The domain $D$ of definition of $u$ will be the domain of unique extendibility.
\end{definition}

\begin{remark}\label{rem2}
{\rm From the discussion above, it follows that $D=(0,D(\alpha))$,
where if
$
D(\alpha)<\infty$,
then $D(\alpha)$ is a double zero of $u$, or  $u'(D(\alpha),\alpha)=0$ and
$u(D(\alpha),\alpha)$ is a relative maximum of $F$ (and  thus $F(u(D(\alpha),\alpha))<0$). }
\end{remark}

 We will denote by $u(\cdot,\alpha)$ such solution. By standard theory of ordinary differential equations, the solution depends continuously on the initial data in any compact subset of its domain of definition,  see for example \cite[Theorem 4.3]{cod}. In the rest of this article, and without further mention,
the constants $\beta^{\pm},\gamma^{\pm},\ \bar\beta$,  $\theta$ and $L_0$ will be as defined in our assumptions $(f_2)$-$(f_4)$.

\begin{proposition}\label{basic} Let $f$ satisfy $(f_1)$-$(f_2)$  and let $u(\cdot,\alpha)$ be a solution of \eqref{ivp}.
\begin{enumerate}
\item[(i)] There exists $C(\alpha)>0$ such that $|u(r,\alpha)|\le C(\alpha)$.
\item[(ii)]
If $u(\cdot,\alpha)$ is   defined in $[0,\infty)$ and $\lim_{r\to \infty}u(r,\alpha)=\ell$, then
$$\quad \lim_{r\to \infty}u'(r,\alpha)=0\quad \mbox{and}\quad \ell\ \mbox{is a zero of $f$}.$$

\end{enumerate}
\end{proposition}
\begin{proof}
Let $u(r)=u(r,\alpha)$ be any solution of \eqref{ivp}. From \eqref{funct-10} and \eqref{I'0}, and noting that from $(f_2)$  $F$ is bounded below by
\be\label{defbarF}-\bar F=\min_{s\in[\beta^-,\beta^+]}F(s),\quad \bar F>0,
\ee
 we have that
$$F(\alpha)\ge F(u(r))\ge -\bar F$$
and thus $(i)$ follows from $(f_2)(ii)$ and the fact that from $(f_2)(i)$, $F$ is strictly increasing in $(\beta^+,\gamma^+)$ and strictly decreasing in $(\gamma^-,\beta^-)$.

Assume next $\lim_{r\to\infty}u(r)=\ell$.
Then from $(i)$ $\ell$ is finite,  and as $I(\cdot,\alpha)$ is decreasing and bounded below by $-\bar F$, we get that
$ \lim_{r\to\infty}u'(r)$ exists. As $u'$ is integrable, we must have $ \lim_{r\to\infty}u'(r)=0$.
   Moreover, from the equation in \eqref{eq2} and applying L'H\^opital's rule twice, we conclude that
\begin{eqnarray*}
 0=\lim_{r\to\infty}\frac{u(r)-\ell}{r^{m'}}&=& -\lim_{r\to\infty}\frac{r^{\frac{N-1}{m-1}}|u'(r)|}{m'r^{\frac{N-1}{m-1}}r^{m'-1}}\\
 &=&
 -\frac{1}{m'}\Bigl(\lim_{r\to\infty}\frac{r^{N-1}|u'(r)|^{m-1}}{r^N}\Bigr)^{m'-1}\\
 &=&
 -\frac{1}{m'}\Bigl(\lim_{r\to\infty}\frac{r^{N-1}f(u(r))}{Nr^{N-1}}\Bigr)^{m'-1}=-\frac{1}{m'}\Bigl(\frac{f(\ell)}{N}\Bigr)^{m'-1},
 \end{eqnarray*}
 and $(ii)$ follows.
\end{proof}

It can be seen that for $\alpha\in[\beta^+,\gamma^+)$, one has $u(r,\alpha)>0$ and $u'(r,\alpha)<0$ for
$r$ small enough, and thus  we can define the extended real number
$$Z_1(\alpha):=\sup\{r\in(0,D(\alpha))\ |\ u(s,\alpha)>0\mbox{ and }u'(s,\alpha)<0\ \mbox{ for all }s\in(0,r)\}.$$
In what follows, we will denote
$$u(Z_i(\alpha),\alpha)=\lim_{r\uparrow Z_i(\alpha)}u(r,\alpha),\ u'(Z_i(\alpha),\alpha)=\lim_{r\uparrow Z_i(\alpha)}u'(r,\alpha),\
u(T_i(\alpha),\alpha)=\lim_{r\uparrow T_i(\alpha)}u(r,\alpha).$$
We set
\begin{eqnarray*}
{\mathcal N_1}&=&\{\alpha\in[\beta^+,\gamma^+)\ :\ u(Z_1(\alpha),\alpha)=0\quad\mbox{and}\quad u'(Z_1(\alpha),\alpha)<0\}\\
{\mathcal G_1}&=&\{\alpha\in[\beta^+,\gamma^+)\ :\ u(Z_1(\alpha),\alpha)=0\quad\mbox{and}\quad u'(Z_1(\alpha),\alpha)=0\}\\
{\mathcal P_1}&=&\{\alpha\in[\beta^+,\gamma^+)\ :\ u(Z_1(\alpha),\alpha)>0\}.
\end{eqnarray*}
From \cite{gst} ${\mathcal N_1}\not=\emptyset$.
Let
$$\widetilde{\mathcal F}_2=\{\alpha\in\mathcal N_1\ :\ u'(r,\alpha)\le 0\quad\mbox{for all }r\in(Z_1(\alpha),D(\alpha))\}.$$
For $\alpha\in\mathcal N_1\setminus\widetilde{\mathcal F}_2$ we define
$${\color{black}T_1(\alpha):=\sup\{r\in(Z_1(\alpha),D(\alpha))\ :\ u'(r,\alpha)\le0\},\quad {U}_1(\alpha)=u(T_1(\alpha),\alpha).}$$
 Also, for $\alpha \in \mathcal N_1\setminus \widetilde{\mathcal F}_2$ we can define the extended real number
$$Z_2(\alpha):=\sup\{r\in(T_1(\alpha),D(\alpha))\ |\ u(s,\alpha)<0\mbox{ and }u'(s,\alpha)>0\ \mbox{ for all }s\in(T_1(\alpha),r)\}.$$

Let now
$${\mathcal F_2}=\{\alpha\in\mathcal N_1\setminus \widetilde{\mathcal F}_2\ :\ u(Z_2(\alpha),\alpha)<0\},$$
\begin{eqnarray*}
{\mathcal N_2}&=&\{\alpha\in\mathcal N_1\setminus \widetilde{\mathcal F}_2\ :\ u(Z_2(\alpha),\alpha)=0\quad\mbox{and}\quad u'(Z_2(\alpha),\alpha)>0\},\\
{\mathcal G_2}&=&\{\alpha\in\mathcal N_1\setminus \widetilde{\mathcal F}_2\ :\ u(Z_2(\alpha),\alpha)=0\quad\mbox{and}\quad u'(Z_2(\alpha),\alpha)=0\},\\
{\mathcal P_2}&=&\widetilde{\mathcal F}_2\cup
{\mathcal F_2}.
\end{eqnarray*}

For $k\ge 3$, and if ${\mathcal N_{k-1}}\not=\emptyset$, we set
$$\widetilde{\mathcal F}_k=\{\alpha\in\mathcal N_{k-1}\ :\ (-1)^ku'(r,\alpha)\le 0\quad\mbox{for all }r\in(Z_{k-1}(\alpha),D(\alpha))\}.$$
For $\alpha\in \mathcal N_{k-1}\setminus\widetilde{\mathcal F}_k$, we set
{\color{black}\ben
T_{k-1}(\alpha):&=&\sup\{r\in(Z_{k-1}(\alpha),D(\alpha))\ :\ (-1)^ku'(r,\alpha)\le 0\},\\ {U}_{k-1}(\alpha):&=&u(T_{k-1}(\alpha),\alpha).
\een
}
 Next, for $\alpha\in \mathcal N_{k-1}\setminus \widetilde{\mathcal F}_k$, we define the extended real number
\begin{eqnarray*}
Z_k(\alpha):=\sup\{r\in(T_{k-1}(\alpha),D(\alpha))\ |\ (-1)^ku(s,\alpha)<0\mbox{ and }(-1)^ku'(s,\alpha)>0\ \\
\mbox{ for all }s\in(T_{k-1}(\alpha),r)\}.
\end{eqnarray*}
 Finally we set
$${{\mathcal F}_k}=\{\alpha\in\mathcal N_{k-1}\setminus \widetilde{\mathcal F}_k\ :\ (-1)^ku(Z_k(\alpha),\alpha)<0\},$$
\begin{eqnarray*}
{\mathcal N_k}&=&\{\alpha\in\mathcal N_{k-1}\setminus \widetilde{\mathcal F}_k\ :\ u(Z_k(\alpha),\alpha)=0\quad\mbox{and}\quad (-1)^ku'(Z_k(\alpha),\alpha)>0\},\\
{\mathcal G_k}&=&\{\alpha\in\mathcal N_{k-1}\setminus \widetilde{\mathcal F}_k\ :\ u(Z_k(\alpha),\alpha)=0\quad\mbox{and}\quad u'(Z_k(\alpha),\alpha)=0\},\\
{\mathcal P_k}&=&\widetilde{\mathcal F}_k\cup
{{\mathcal F}_k}.
\end{eqnarray*}

Concerning the sets ${\mathcal N_k}$ and ${\mathcal P_k}$ we have:

\begin{proposition}
\label{open-sets}
The sets ${\mathcal N_k}$ and ${\mathcal P_k}$ are open in $[\beta^+,\gamma^+)$ and $\beta^+\in\mathcal P_1$.
\end{proposition}
\begin{proof}
The proof that ${\mathcal N_k}$ is open is by continuity. Indeed,
let $\bar\alpha\in{\mathcal N_k}$. We may assume that $u(r,\bar\alpha)$ is decreasing in $(T_{k-1}(\bar\alpha),Z_1(\bar\alpha))$,
$u(Z_k(\bar\alpha),\bar\alpha)=0$ and $u'(Z_k(\bar\alpha),\bar\alpha)<0$ (If $k=1$, we set $T_0(\bar\alpha)=0$).

Since
$u'(Z_1(\bar\alpha),\bar\alpha)<0$ we can extend the solution $u(r,\bar\alpha)$ to an interval
$$[T_{k-1}(\bar\alpha),Z_k(\bar\alpha)+\varepsilon]$$ and $u'(r,\bar\alpha)<0$ for all $r\in(T_{k-1}(\bar\alpha),Z_k(\bar\alpha)+\varepsilon]$ and the result follows by continuous dependence of the solution on the initial condition.

\medskip

The proof that $\beta^+\in\mathcal P_1$ and $\mathcal P_1$ is open in $[\beta^+,\gamma^+)$ can be found in \cite{gst}. Let $k>1$. The proof that $\mathcal P_k$ is open is based in the fact that the functional $I$ defined in \eqref{funct-10} is
decreasing in $r$, and $\alpha\in\mathcal P_k$ if and only if $\alpha\in \mathcal N_{k-1}$ and $I(r_1,\alpha)<0$ for some $r_1\in(0,Z_{k}(\alpha))$. Hence the openness of ${\mathcal P_k}$ follows
from the
continuous dependence of solutions to \eqref{ivp} in the initial value $\alpha$ and from the openness of ${\mathcal N_{k-1}}$.

 Let $\bar\alpha\in \mathcal P_k$ and assume first that $Z_{k}(\bar\alpha)=\infty$.
 From Proposition \ref{basic},
 $$\lim_{r\to \infty}I(r,\bar\alpha)=F(\ell)<0.$$

Assume next $Z_k(\bar\alpha)<\infty$ and $Z_{k}(\bar\alpha)=T_{k-1}(\bar\alpha)$. Then   $I(Z_{k}(\bar\alpha),\bar\alpha)<0$, see Remark \ref{rem2}.

 The last possibility is $Z_k(\bar\alpha)<\infty$ and  $Z_{k}(\bar\alpha)$ is a maximum point for $u$ with $u(Z_k(\bar\alpha),\bar\alpha)<0$, or a minimum point of $u$ with $u(Z_k(\bar\alpha),\bar\alpha)>0$,  implying that either
$$0\le -(\phi_m(u'))'(Z_{k}(\bar\alpha),\bar\alpha)=f(u(Z_{k}(\bar\alpha),\bar\alpha))$$
and hence $\beta^-<u(Z_{k}(\bar\alpha),\bar\alpha)<0$, or
$$0\ge -(\phi_m(u'))'(Z_{k}(\bar\alpha),\bar\alpha)=f(u(Z_{k}(\bar\alpha),\bar\alpha))$$
and thus $0<u(Z_{k}(\bar\alpha),\bar\alpha)<\beta^+$.
 Therefore in both cases
 $$I(Z_{k}(\bar\alpha),\bar\alpha)=F(u(Z_{k}(\bar\alpha),\bar\alpha))<0.$$

Conversely, if $\alpha\not\in{\mathcal P_k}$ and $\alpha\in{\mathcal N_{k-1}}$, then $\alpha\in{\mathcal G_k}\cup{\mathcal N_k}$, and thus  $I(r,\alpha)\ge I(Z_{k}(\alpha),\alpha)\ge 0$ for all $r\in (0, Z_{k}(\alpha)).$

\end{proof}

The following lemma is an extension of \cite[Lemma 3.1]{gst}.
\begin{lemma}\label{gstlema}
Let $\alpha\in\mathcal N_k$, $T_k(\alpha)<\infty$, and $U_k(\alpha)\in(\bar\beta,\gamma^+)$ ($U_k(\alpha)\in(\gamma^-,-\bar\beta)$). Let $\bar r\ge T_k(\alpha)$ be the first point after $T_k(\alpha)$ at which $u(\bar r,\alpha)=\bar\beta$ ($u(\bar r,\alpha)=-\bar\beta$). If
\be\label{lema31}\bar r\ge \bar C:=(N-1)(m')^{1/m'}\bar\beta\max\Bigl\{\frac{(\bar F+F(\bar\beta))^{1/m'}}{F(\bar\beta)},\frac{(\bar F+F(-\bar\beta))^{1/m'}}{F(-\bar\beta)}\Bigr\},
\ee
then $\alpha\in \mathcal N_{k+1}$.
\end{lemma}
\begin{proof}
The proof of this result is exactly the same as that of Lemma 3.1 in \cite{gst}, with $b=\bar\beta$,  and $R=Z_{k+1}$ and thus we omit it.
\end{proof}

Finally in this section we state a useful and well known Pohozaev type identity
(see \cite{p})
which plays a key role in our proofs.
For a solution $u(\cdot,\alpha)$ of \eqref{ivp}, set
$$E(r,\alpha):=mr^NI(r,\alpha)+(N-m)r^{N-1}\phi_m(u'(r,\alpha))u(r,\alpha).$$
Then
\be\label{energy}E'(r,\alpha)=r^{N-1}Q(u(r,\alpha)),
\ee
and thus, in any interval $[R_1,R_2]$ where $uu'\le 0$ and $u'(R_1,\alpha)=0$ it holds that
\be\label{energy2}
mr^NI(r,\alpha)-mR_1^{N-1}F(u(R_1,\alpha))\ge\int_{R_1}^rt^{N-1}Q(u(t,\alpha))dt.
\ee

\section{Existence of bound states, proof of Theorem \ref{main} }\label{exist}
The proof of our main result is based in the two crucial lemmas below.

\begin{lemma}\label{nn2} Assume that $f$ satisfies $(f_1)$-$(f_4)$. Then,  for each $k\in\mathbb N$, there exists  $\alpha_k\in(\beta^+,\gamma^+)$ with $\alpha_{k}<\alpha_{k+1}$ and $[\alpha_k,\gamma^+)\subset{\mathcal N}_k$.
\end{lemma}
\begin{proof}
 If $\gamma^+<\infty$, then the result for $k=1$ follows from the proof of Claim 3 case $(C1)$, in \cite[Theorem 1]{gst} Moreover, in this case the authors prove that $(f_4)(a)$ implies that
  $$\lim_{\alpha\to\gamma^+}r(\bar\beta,\alpha)=\infty,$$ where $r(\bar\beta,\alpha)$ is the first value of $r>0$ such that $u(r,\alpha)=\bar\beta$.

   In the case $\gamma^+=\infty$, the result for $k=1$ follows due to our stronger assumption $(f_3)(a)(ii)$. Indeed, in the proof of Theorem 1, Claim 3 in \cite{gst}, the authors prove that if for some $\theta\in(0,1)$,
 $$\inf_{s_1,s_2,\alpha}Q(s_2)\Bigl(\frac{|\alpha|^{m-2}\alpha}{f(s_1)}\Bigr)^{N/m}\Bigr)\ge C_0,\quad \forall s_1, s_2\in[\theta\alpha,\alpha]$$
where $C_0>0$ is a fixed constant, then $\alpha\in\mathcal N_1$. ($\mathcal N_1=I^-$ in their notation). By our assumption $(f_3)(ii)$, there exists $\alpha_1>0$ such that for all $\alpha\ge\alpha_1$ we have
$$\inf_{s_1,s_2,\alpha}Q(s_2)\Bigl(\frac{|\alpha|^{m-2}\alpha}{f(s_1)}\Bigr)^{N/m}\Bigr)\ge C_0,\quad \forall s_1, s_2\in[\theta\alpha,\alpha]$$
and thus $[\alpha_1,\infty)\subset{\mathcal N}_1$. Clearly, we can assume that $\alpha_1\ge\bar\beta$.

 Let now $k\ge 2$. We will prove our result inductively, starting with $\alpha_1>\beta^+$  such that $[\alpha_1,\gamma^+)\subset{\mathcal N}_1$.
\medskip

\noindent{\bf Claim 1.} There exists   $\bar\alpha\ge\alpha_1$,  such that $T_1(\alpha)<\infty$ for all $\alpha\in[\bar\alpha,\gamma^+)$.

Assume that this is not true, and hence there exists a sequence $\{\alpha_i\}_{i\ge 2}$, $\alpha_i\to\gamma^+$ as $i\to\infty$, such  $T_1(\alpha_i)=\infty$ for all $i\ge2$. Then $u(\cdot,\alpha_i)$ is decreasing,  implying, see Proposition \ref{basic}, that $\lim_{r\to\infty}u(r,\alpha_i)=\ell_i$, $|\ell_i|<|\beta^-|$, see Figure \ref{fig1m}.

As
$$|u'(r,\alpha_i)|^m=|u'(r,\alpha_i)|^m+m'F(u(r,\alpha_i))-m'F(u(r,\alpha_i))\le m'I(r,\alpha_i)+m'\bar F,$$
where $\bar F$ is defined in \eqref{defbarF}, we have that
$$\sup_{r\ge r_i}|u'(r,\alpha_i)|\le m'^{1/m}(I(r,\alpha_i)+\bar F)^{1/m},$$
and as $u'(r,\alpha_i)\to0$ as $r\to\infty$, $I(r,\alpha_i)\to L_i=F(\ell_i)<0$ with $|L_i|\le \bar F$ as $r\to\infty$.

Denote by $r_i:=r(\bar\beta,\alpha_i)$ the first value of $r$ where $u$ takes the value $\bar\beta$. By integrating \eqref{I'0} over $[r_i,\infty)$, we obtain
\be\label{n33}
I(r_i,\alpha_i)-L_i&=&(N-1)\int_{r_i}^\infty\frac{|u'(r,\alpha_i)|^m}{r}dr\nonumber\\
&\le&(N-1)(m')^{1/m'}\frac{(I(r_i,\alpha_i)+\bar F)^{1/m'}}{r_i}(|\ell_i|+\bar\beta)
\ee
implying that
\be\label{imp1}
\frac{r_i(I(r_i,\alpha_i)-L_i)}{(I(r_i,\alpha_i)+\bar F)^{1/m'}}\le (N-1)(m')^{1/m'}(|\beta^-|+\bar\beta).
\ee
We will show that \eqref{imp1} is not possible.

We have two possibilities: either {\bf (a) }$r_i$ is bounded, or {\bf (b) }$r_i\to\infty$ through some subsequence. We recall that in case $\gamma^+<\infty$, only {\bf (b) } is possible.

\noindent Case {\bf (a)}. In this case necessarily $\gamma^+=\infty$. We assume that there is $M>0$ such that $r_i\le M$ for all $i$.

\begin{figure}[h]
\begin{center}
 \includegraphics[keepaspectratio, width=12cm]{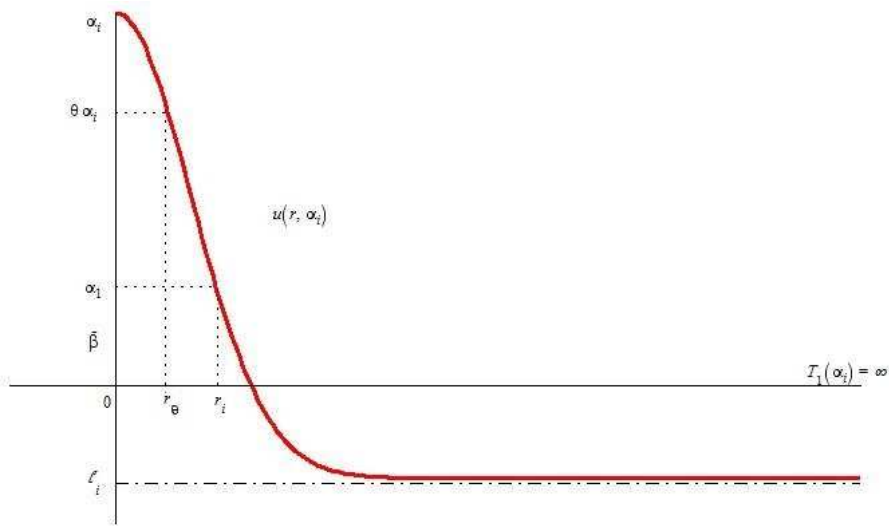}
 \end{center}
 \caption{The shape of a solution satisfying $T_1(\alpha_i)=\infty$}
 \label{fig1m}
 \end{figure}

Let $\theta\in(0,1)$ be as in assumption $(f_3)(ii)$ and let $i$ be large enough to have $\theta\alpha_i>\bar\beta$. By setting  $r_\theta >0$  the point where $u(r_\theta,\alpha_i)=\theta \alpha_i$, integration of \eqref{energy} over $[0,r_i]$ yields
\be\label{copia1}
mr_i^NI(r_i,\alpha_i)&\ge& (\int_{0}^{r_\theta}+\int_{r_\theta}^{r_i})t^{N-1}Q(u(t,\alpha_i))dt\nonumber\\
 &\ge& \int_{0}^{r_\theta}t^{N-1}Q(u(t,\alpha_i))dt\quad\mbox{ (as $Q(u(t,\alpha_i))\ge 0 $ in $[r_\theta,r_i]$) }\\
&\ge& Q(s_2)\frac{r_\theta^N}{N}\quad\mbox{ where we have set $Q(s_2)=\min_{s\in[\theta\alpha_i,\alpha_i]}Q(s)$.}\nonumber
\ee
Now we estimate $r_\theta$:
Set $f(s_1)=\max_{s\in[\theta \alpha_i,\alpha_i]}f(s)$ ($s_1\in [\theta \alpha_i,\alpha_i]$). From the equation in \eqref{ivp},
\ben
-r^{N-1}\phi_m(u'(r,\alpha_i))=\int_{0}^rt^{N-1}f(u(t,\alpha_i)dt\le f(s_1)\frac{r^N}{N},
\een
hence
\ben
-u'(r,\alpha_i)\le \phi_{m'}(f(s_1))\frac{r^{m'-1}}{N^{m'-1}}.
\een
Integrating now this last inequality over $[0,r_\theta]$ we obtain
$$(1-\theta)\alpha_i\le \phi_{m'}(f(s_1))\frac{r_\theta^{m'}}{m'N^{m'-1}}$$
and thus
$$r_\theta\ge \Bigl(\frac{c\alpha_i^{m-1}}{f(s_1)}\Bigr)^{1/m},$$
where $c=((1-\theta)m')^{m-1}N$. Therefore, by $(f_3)(ii)$ we conclude that
$$mr_i^NI(r_i,\alpha_i)\ge \frac{1}{N}Q(s_2)\Bigl(\frac{c\alpha_i^{m-1}}{f(s_1)}\Bigr)^{N/m}\to\infty\mbox{ as $i\to\infty$,}$$
 hence, as $N\ge m$ and by the boundedness assumption on $r_i$, we obtain that
$$r_i^mI(r_i,\alpha_i)\to\infty\quad\mbox{as }i\to\infty$$
contradicting \eqref{imp1}.
\medskip

\noindent Case {\bf (b)}. From \eqref{imp1},
\be\label{imp2}
0\le \frac{I(r_i,\alpha_i)-L_i}{(I(r_i,\alpha_i)+\bar F)^{1/m'}}\le \frac{N-1}{r_i}(m')^{1/m'}(|\beta^-|+\bar\beta).
\ee
We note first that in this case $I(r_i,\alpha_i)\to\infty$ as $i\to\infty$. Indeed, if for some subsequence (still renamed the same) we have that $I(r_i,\alpha_i)$ is bounded, then from \eqref{imp2} it must be that
$$I(r_i,\alpha_i)-L_i\to0\quad\mbox{as }i\to\infty,$$
which is a contradiction because
$$I(r_i,\alpha_i)-L_i\ge F(\bar\beta)-L_i\ge F(\bar\beta)>0.$$
Hence $I(r_i,\alpha_i)\to\infty$ as $i\to\infty$. Then as
$$\frac{I(r_i,\alpha_i)-L_i}{(I(r_i,\alpha_i)+\bar F)^{1/m'}}=(I(r_i,\alpha_i)+\bar F)^{1/m}-\frac{\bar F+L_i}{(I(r_i,\alpha_i)+\bar F)^{1/m'}},$$
we get from \eqref{imp2} that
$$\bar F^{1/m}<(I(r_i,\alpha_i)+\bar F)^{1/m}\le \frac{\bar F+L_i}{(I(r_i,\alpha_i)+\bar F)^{1/m'}}+ \frac{N-1}{r_i}(m')^{1/m'}(|\beta^-|+\bar\beta)\to 0,$$
also a contradiction
and our Claim 1 follows.
\medskip

\noindent{\bf Claim 2.}  $\lim_{\alpha\to\gamma^+}u(T_1(\alpha),\alpha)=\gamma^-$.

We will argue by contradiction and thus assume that there exists  $\gamma^-<M<\beta^-$  and a sequence  $\{\alpha_i\}\to\gamma^+$ such that
$U_1(\alpha_i)=u(T_1(\alpha_i),\alpha_i)\ge M.$ From Claim 1 we may assume that $T_1(\alpha_i)<\infty$ for all $i$.

 Let us denote by $r(\cdot,\alpha_{i})$ the inverse of $u(\cdot,\alpha_{i})$ in the interval $[0,T_1(\alpha_{i})]$. From $(f_2)(i),(ii)$, there exists $\tilde M\in(\beta^+,\gamma^+)$ such that $F(\tilde M)-F(M)>0$. Integrating now $I'$ over $[r(\tilde M,\alpha_{i}),T_1(\alpha_{i})]$ we obtain
\ben
I(r(\tilde M,\alpha_{i}),\alpha_{i})-F(U_1(\alpha_{i}))&=&(N-1)\int_{r(\tilde M,\alpha_{i})}^{T_1(\alpha_{i})}\frac{|u'(r,\alpha_{i})|^m}{r}dr\\
&\le&(N-1)(m')^{\frac{1}{m'}}\frac{(I(r(\tilde M,\alpha_{i}),\alpha_{i})+\bar F)^{\frac{1}{m'}}}{r(\tilde M,\alpha_{i})}(|M|+\tilde M),
\een
hence
$$I(r(\tilde M,\alpha_{i}),\alpha_{i})-F(M)\le (N-1)(m')^{\frac{1}{m'}}\frac{(I(r(\tilde M,\alpha_{i}),\alpha_{i})+\bar F)^{\frac{1}{m'}}}{r(\tilde M,\alpha_{i})}(|M|+\tilde M).$$
We can prove that this inequality is not possible by using the same arguments we used to prove that \eqref{imp1} led to a contradiction, hence our claim follows.

\medskip

\noindent{\bf Claim 3.} There exists $\alpha_2>\alpha_1$ such that  $[\alpha_2,\gamma^+)\subset\mathcal N_2$.

We argue by contradiction  assuming that there exists a sequence $\{\alpha_i\}\to\gamma^+$ such that $\alpha_i\in{\mathcal G}_2\cup {\mathcal P}_2$ for all  $i$ and study the solutions $u(\cdot, \alpha_i)$ in the interval $(0,Z_2(\alpha_i))$. Our argument follows the ideas of the proof of \cite[Theorem 1]{gst}.

From Claim 1  we may assume that $T_1=T_1(\alpha_i)<\infty$.
\smallskip

\noindent {\bf Case $\gamma^-=-\infty$:}
In this case from Claim 2, $U_1(\alpha_i)\to-\infty$ as $i\to\infty$. Also,
by $(f_1)$ and $(f_3)(b)(i)$, $Q$ is bounded below, and as $Q(\beta^{-})=-(N-m)\beta^{-}f(\beta^{-})<0$,  we can set
$$0<\bar Q=-\inf_{s<0}Q(s).$$
 Let $s_\theta(\alpha_i)> T_1(\alpha_i)$ be now the point where $u(s_\theta,\alpha_i)=\theta U_1(\alpha_i)$, where by Claim 2 we may assume that $-\theta U_1(\alpha_i)>\bar\beta$ for all $i$, and let $s_i> T_1(\alpha_i)$ be the first point after $T_1(\alpha_i)$ where $u(s_i,\alpha_i)=-\bar\beta$ (and thus from Lemma \ref{gstlema}, $s_i<\bar C$), see Figure \ref{f2} below.
\begin{figure}[h]
\begin{center}
 \includegraphics[keepaspectratio, width=8cm]{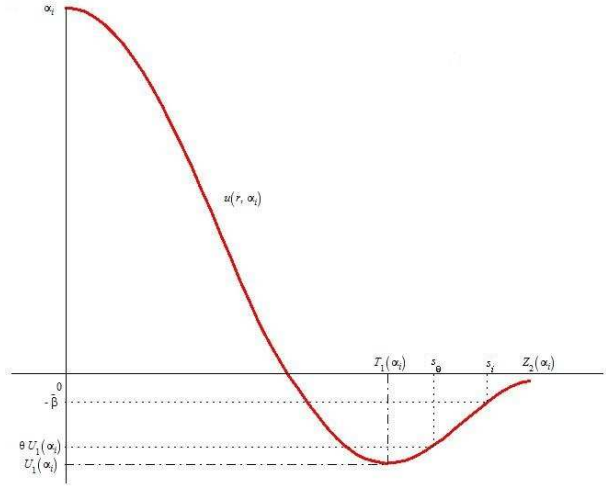}
 \end{center}
 \caption{The shape of a solution such that $\alpha_i\in{\mathcal G}_2\cup {\mathcal P}_2$}
 \label{f2}
 \end{figure}

Then, for any $r\in(T_1(\alpha_i),Z_2(\alpha_i))$, by \eqref{energy2} we have
\ben
mr^NI(r,\alpha_i)-mT_1^NF(U_1(\alpha_i))&\ge& (\int_{T_1}^{s_\theta}+\underbrace{\int_{s_\theta}^{s_i}}_{\ge 0}+\int_{s_i}^r)t^{N-1}Q(u(t,\alpha_i))dt\\
&\ge& \int_{T_1}^{s_\theta}t^{N-1}Q(u(t,\alpha_i))dt-\bar Q \frac{r^N}{N}\\
&\ge& Q(s_2)\Bigl(\frac{s_\theta^N-T_1^N}{N}\Bigr)-\bar Q \frac{r^N}{N}\\
&\ge& Q(s_2)\frac{s_\theta^N}{N}-T_1^N\frac{NmF(s_2)}{N}-\bar Q \frac{r^N}{N}
\een
for some $s_2\in[U_1(\alpha_i),\theta U_1(\alpha_i)]$.
 $s_\theta$ can now be estimated as we did before but setting now
 $f(s_1)=\min_{s\in[U_1(\alpha_i),\theta U_1(\alpha_i)]}f(s)$ ($s_1\in [U_1(\alpha_i),\theta U_1(\alpha_i)]$). We obtain that
\be\label{stheta}s_\theta\ge \Bigl(\frac{c|U_1(\alpha_i)|^{m-1}}{|f(s_1)|}\Bigr)^{1/m},
\ee
where $c=((1-\theta)m')^{m-1}N$.
Therefore, for $r\in(T_1(\alpha_i),Z_2(\alpha_i))$ we have
\be\label{ener-acot}
mr^NI(r,\alpha_i)&\ge&Q(s_2)\frac{s_\theta^N}{N}-T_1^N\frac{NmF(s_2)}{N}+mT_1^NF(U_1(\alpha_i))-\bar Q \frac{r^N}{N}\nonumber\\
&\ge&\frac{1}{N}Q(s_2)\Bigl(\frac{c|U_1(\alpha_i)|^{m-1}}{|f(s_1)|}\Bigr)^{N/m}-mT_1^N \underbrace{(F(s_2)-F(U_1(\alpha_i)))}_{\le 0}-\bar Q\frac{r^N}{N}\nonumber\\
&\ge&\frac{1}{N}Q(s_2)\Bigl(\frac{c|U_1(\alpha_i)|^{m-1}}{|f(s_1)|}\Bigr)^{N/m}-\bar Q\frac{r^N}{N}.
\ee
As in \cite{gst}, for an arbitrary $a>0$ we set $R_a=\min\{\bar C+a, Z_2(\alpha_i)\}$. From Lemma \ref{gstlema}, $R_a\in(s_i, Z_2(\alpha_i))$. Also, by definition $R_a\le \bar C+a$ which is finite and independent of $\alpha_i$, hence, by $(f_3)(ii)$, we can choose $i_0$
\be\label{ener3}
I(r,\alpha_i)\ge F(-\bar\beta)+\frac{1}{m'}\Bigl(\frac{\bar\beta}{a}\Bigr)^m\quad\mbox{for all }r\in(s_i,R_a)\quad\mbox{and }i\ge i_0.
\ee
As in \cite{gst}, we prove that for these $\alpha_i$, it holds that $R_a=\bar C+a$. Indeed, if not, $R_a=Z_2(\alpha_i)$ and, as $\alpha_i\in {\mathcal G}_2\cup {\mathcal P}_2$, it must be that $u'(R_a,\alpha_i)=0$, implying that $I(R_a,\alpha_i)=F(u(R_a,\alpha_i))<F(-\bar\beta)$ contradicting \eqref{ener3}. But, by \eqref{ener3}, $|u'(r,\alpha_i)|\ge \bar\beta/a$ in $(s_i,R_a)$, hence
$$u(R_a,\alpha_i)-u(s_i,\alpha_i)\ge \frac{\bar\beta}{a}(R_a-s_i)$$
implying
\ben
u(R_a,\alpha_i)&\ge& -\bar\beta+\frac{\bar\beta}{a}(\bar C+a-s_i)\\
&=& \frac{\bar\beta}{a}(\bar C-s_i)>0,
\een
a contradiction. Hence there exists $\alpha_2>\alpha_1$ such that $[\alpha_2,\gamma^+)\subset\mathcal N_2$.
\medskip

\noindent {\bf Case $\gamma^->-\infty$:}

In this case from Claim 2, $U_1(\alpha_i)\downarrow\gamma^-$. By $(f_4)(b)$, there exists $\bar\beta>0$ with $\gamma^-<-\bar\beta<\beta^-$, such that
$$\frac{-f(s)}{(s-\gamma^-)^{m-1}}< L_0\quad\mbox{for all $s\in(\gamma^-,-\bar\beta]$}.$$
Let now  $i_0\in\mathbb N$ be such that
$$\gamma^-<U_1(\alpha_i)<-\bar\beta\quad \mbox{for all $i\ge i_0$}.$$
Denote again by $s_i$ the first point after $T_1(\alpha_i)$ where $u(s_i,\alpha_i),\alpha_i)=-\bar\beta$. By integrating the equation in \eqref{ivp} over $[T_1(\alpha_i),r]$, with $r\in(T_1(\alpha_i),s_i]$, we obtain
\ben
r^{N-1}\phi_m(u'(r,\alpha_i))&\le & L_0\int_{T_1(\alpha_i)}^rt^{N-1}(u(t,\alpha_i)-\gamma^-)^{m-1}dt\\
&\le& \frac{L_0}{N}(u(r,\alpha_i)-\gamma^-)^{m-1}r^N,
\een
hence
\ben
0\le u'(r,\alpha_i)&\le& C(u(r,\alpha_i)-\gamma^-)s_i^{1/(m-1)}\\
&=&C(u(T_1(\alpha_i),\alpha_i)+\int_{T_1(\alpha_i)}^ru'(t,\alpha_i)dt-\gamma^-)s_i^{1/(m-1)},
\een
where $C=(L_0/N)^{m'-1}$. Using Gronwall's inequality we get
$$u'(r,\alpha_i)\le Cs_i^{1/(m-1)}(u(T_1(\alpha_i),\alpha_i)-\gamma^-)\exp(Cs_i^{m/(m-1)}),$$
implying as in \cite{gst} that
\be\label{cc1}
\lim_{i\to\infty}s_i=\infty.
\ee
Using now Lemma \ref{gstlema} we conclude that $\alpha_i\in\mathcal N_2$ for $i$ large enough, which contradicts our initial assumption.
Hence there exists $\alpha_2>\alpha_1$ such that $[\alpha_2,\gamma^+)\subset\mathcal N_2$.

\medskip

\noindent{\bf Claim 4.} There exists $\bar\alpha>\alpha_2$ such that $T_2(\alpha)<\infty$ for all $\alpha\in[\bar\alpha,\gamma^+)$.

Assume as in the proof of Claim 1 that this is not true, and hence there exists a sequence $\{\alpha_i\}_{i\ge 3}$, $\alpha_i\to\gamma^+$ as $i\to\infty$, such  $T_2(\alpha_i)=\infty$ for all $i\ge3$. Then $u(\cdot,\alpha_i)$ is increasing, implying, see Proposition \ref{basic}, that $\lim_{r\to\infty}u(r,\alpha_i)=\ell_i$, $\ell_i<\beta^+$.

Following the proof
 in Claim 1, we denote once more by $s_i:=r(-\bar\beta,\alpha_i)$ the first value of $r$ after $T_1(\alpha_i)$  where $u$ takes the value $-\bar\beta$. By integrating \eqref{I'0} over $[s_i,\infty)$, we obtain
\be\label{imp13}
\frac{s_i(I(s_i,\alpha_i)-L_i)}{(I(s_i,\alpha_i)+\bar F)^{1/m'}}\le (m')^{1/m'}(\beta^++\bar\beta).
\ee
We will show that \eqref{imp13} is not possible. Again we have two possibilities: either {\bf (a) }$s_i$ is bounded, or {\bf (b) }$s_i\to\infty$ through some subsequence. We recall that in case $\gamma^->-\infty$, only {\bf (b) } is possible, see \eqref{cc1}. The proof in case {\bf (b) } is the same as the corresponding one in Claim 1. In case {\bf (a)}, some minor changes are needed: the integral in \eqref{copia1} changes to
\ben
mr^NI(s_i,\alpha_i)-mT_1^NF(U_1(\alpha_i))&\ge& (\int_{T_1}^{s_\theta}+\underbrace{\int_{s_\theta}^{s_i}}_{\ge 0})t^{N-1}Q(u(t,\alpha_i))dt\\
&\ge& \int_{T_1}^{s_\theta}t^{N-1}Q(u(t,\alpha_i))dt\\
&\ge& Q(s_2)\Bigl(\frac{s_\theta^N-T_1^N}{N}\Bigr)\\
&\ge& Q(s_2)\frac{s_\theta^N}{N}-T_1^NmF(s_2),
\een
where $s_\theta(\alpha_i)> T_1(\alpha_i)$ is  the first point after $T_1(\alpha_i)$ where $u(s_\theta,\alpha_i)=\theta U_1(\alpha_i)$ and
$Q(s_2):=\min_{s\in [U_1(\alpha_i),\theta U_1(\alpha_i)]}Q(s)$. As $F(U_1(\alpha_i))-F(s_2)\ge0$, we obtain
$$mr^NI(s_i,\alpha_i)\ge Q(s_2)\frac{s_\theta^N}{N},$$
and thus we can continue arguing as in the proof of Claim 1 using the estimate for $s_\theta$ given in \eqref{stheta}.
\begin{figure}[h]
\begin{center}
 \includegraphics[keepaspectratio, width=8cm]{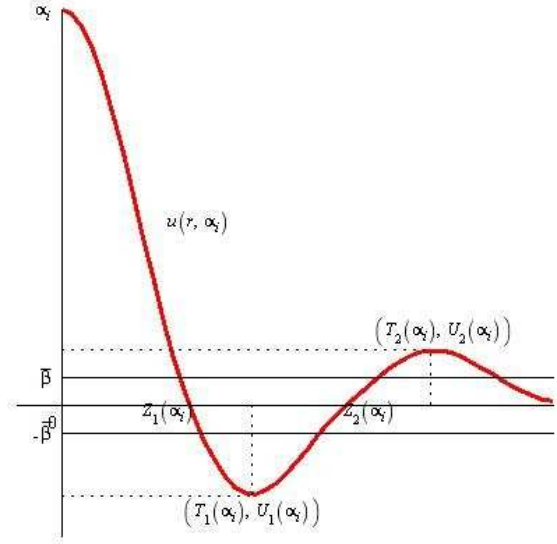}
 \end{center}
 \caption{The shape of $u(r,\alpha_i)$ for $\alpha_i\in \mathcal G_3\cup\mathcal P_3$.}
 \label{f3}
 \end{figure}

\noindent{\bf Proof of the lemma:}

We can repeat the arguments in  Claim 2, to prove that $U_2(\alpha)=u(T_2(\alpha),\alpha)\to\gamma^+$ as $\alpha\to\gamma^+$. Then we argue as in the proof of Claim 3 to show the existence of $\alpha_3>\alpha_2$ such that $[\alpha_3,\infty)\subset\mathcal N_3$, see Figure \ref{f3} above. Repeating the argument as many times as necessary the lemma follows.

\end{proof}
\medskip

In view of this lemma, the natural candidate to belong to $\mathcal G_k$ is
$$\alpha_k^*=\inf\{\alpha\ge\beta^+\ |\ [\alpha,\gamma^+)\subseteq \mathcal N_k\}.$$
In order to prove that $\alpha_k^*\in\mathcal G_k$ we need the following result:
\medskip

\begin{lemma}\label{p2m} Assume that $f$ satisfies $(f_1)$-$(f_2)$, let $k\in\mathbb N$ and let $\bar\alpha\in\mathcal G_k$. If $\{\alpha_i\}\subset\mathcal N_k$ is any sequence  such that $\alpha_i\to\bar\alpha$ as $i\to\infty$, then there exists $i_0$ such that
\be\label{*1m}\alpha_i \in\mathcal P_{k+1}\quad\mbox{for all }i\ge i_0.
\ee
\end{lemma}

\begin{proof}
Let $\bar\alpha\in\mathcal G_k$ and let (see Figure \ref{flema} below)
 $\alpha_i\to\bar\alpha$  with $\alpha_i\in\mathcal N_k$. Assume by contradiction that
$\{\alpha_i\}$ contains a subsequence, still denoted the same, such that $\{\alpha_i\}\subset {\mathcal N}_{k+1}\cup{\mathcal G}_{k+1}$.
Then $T_k(\alpha_i)<\infty$ for all $i$, and since $I(Z_{k+1}(\alpha_i),\alpha_i)\ge 0$, it follows that $I(T_{k}(\alpha_i),\alpha_i)> 0$.
Without loss of generality we may assume that $u(\cdot,\bar\alpha)$ is decreasing in $(T_{k-1}(\bar\alpha),Z_k(\bar\alpha))$ so that $T_k(\alpha_i)$ is a minimum point for $u(\cdot,\alpha_i)$ and therefore
\be\label{n1} u(T_k(\alpha_i),\alpha_i)<\beta^-.
\ee

\begin{figure}[h]
\begin{center}
 \includegraphics[keepaspectratio, width=8cm]{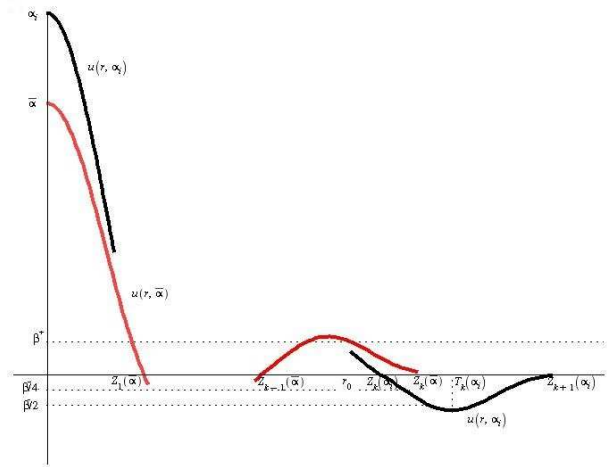}
 \end{center}
 \caption{Situation in Lemma \ref{p2m}}
 \label{flema}
 \end{figure}

Let us denote by $r(\cdot,\alpha_i)$ the inverse of $u(\cdot,\alpha_i)$ in
$(T_{k-1}(\alpha_i),T_k(\alpha_i))$ and let now $\varepsilon\in(0,1)$. Since
$$\lim_{r\to Z_k(\bar\alpha)}I(r,\bar\alpha)=0\quad\mbox{and}\quad \lim_{r\to Z_k(\bar\alpha)}u(r,\bar\alpha)=0,$$
there exists $r_0>T_{k-1}(\bar\alpha)$ such that
$$I(r_0,\bar\alpha)<\varepsilon,\quad 0<u(r_0,\bar\alpha)<\beta^+/2,$$
and therefore by continuity, for  $i$ large enough, $0<u(r_0,\alpha_i)<\beta^+$, $Z_k(\alpha_i)>r_0$, and
$$I(r_0,\alpha_i)<2\varepsilon. $$
Since $I$ is decreasing in $r$, we have that
$$I(r,\alpha_i)<2\varepsilon\quad\mbox{for all}\quad r\in (r_0,Z_{k+1}(\alpha_i)),$$
hence
\be\label{n2}
|u'(r,\alpha_i)|\le \sqrt[m]{\bar F+2}:=K\quad\mbox{for all }r\in (r_0,Z_{k+1}(\alpha_i))
\ee
and $i$ large enough. From \eqref{n1}, $[\beta^-,0]\subset[u(T_k(\alpha_i),\alpha_i),0]$, and from \eqref{n2}, by the mean value theorem we obtain that
\be\label{n3}r\Bigl(\frac{\beta^-}{2},\alpha_i\Bigr)-r\Bigl(\frac{\beta^-}{4},\alpha_i\Bigr)\ge \frac{|\beta^-|}{4K}.
\ee
 Let now
$$H(r,\alpha)=r^{m'(N-1)}I(r,\alpha).$$
Then
$$H'(r,\alpha)=m'(N-1)r^{m'(N-1)-1}F(u(r,\alpha)),$$
implying that for $\alpha=\bar\alpha$, $H'(r,\bar\alpha)<0$ for all $r\in(r_0, Z_k(\bar\alpha))$ and
$$H(r,\bar\alpha)\downarrow L\ge 0$$
as $r\to Z_k(\bar\alpha)$. Also, by choosing a larger $r_0$ if necessary, we may assume $H(r_0,\bar\alpha)<L+\varepsilon$. Thus by continuity we may assume that
$$H(r_0,\alpha_i)\le L+2\varepsilon\quad\mbox{for $i$ large enough.}$$
Also, as $u(r,\alpha_i)<\beta^+$ for $r\in[r_0,Z_k(\alpha_i)]$, $H$ is decreasing in $[r_0,Z_k(\alpha_i)]$ implying
$$H(Z_k(\alpha_i),\alpha_i)\le L+2\varepsilon\quad\mbox{for $i$ large enough.}$$
Integrating $H'(\cdot,\alpha_i)$ over $(Z_k(\alpha_i),r(\frac{\beta^-}{2},\alpha_i))$, we find that
$$
H(r\Bigl(\frac{\beta^-}{2},\alpha_i\Bigr),\alpha_i)-H(Z_k(\alpha_i),\alpha_i)=-m'(N-1)
\int_{Z_k(\alpha_i)}^{r(\frac{\beta^-}{2},\alpha_i)}t^{m'(N-1)-1}|F(u(t,\alpha_i))|dt$$
and thus, observing that since $N\ge m$, we have $m'(N-1)-1\ge m-1>0$ implying
\ben
H(r\Bigl(\frac{\beta^-}{2},\alpha_i\Bigr),\alpha_i)&\le& L+2\varepsilon-m'(N-1)(Z_k(\alpha_i))^{m'(N-1)-1}\int_{Z_k(\alpha_i)}^{r(\frac{\beta^-}{2},\alpha_i)}|F(u(t,\alpha_i))|dt\\
&\le& L+2\varepsilon-m'(N-1)(Z_k(\alpha_i))^{m'(N-1)-1}\int_{r(\frac{\beta^-}{4},\alpha_i)}^{r(\frac{\beta^-}{2},\alpha_i)}|F(u(t,\alpha_i))|dt\\
&\le& L+2\varepsilon-m'(N-1)(Z_k(\alpha_i))^{m'(N-1)-1}(r\Bigl(\frac{\beta^-}{2},\alpha_i\Bigr)-r\Bigl(\frac{\beta^-}{4},\alpha_i\Bigr))C\\
&\le& L+2\varepsilon-m'(N-1)(Z_k(\alpha_i))^{m'(N-1)-1}\frac{b}{4K}C,
\een
where $C:=\inf\{|F(s)|,\ s\in[\frac{\beta^-}{2},\frac{\beta^-}{4}]\}$. If $Z_k(\bar\alpha)=\infty$, by taking $i$ larger if necessary, we conclude that $H(r(\frac{\beta^-}{2},\alpha_i),\alpha_i)<0$ and thus $\alpha_i\in\mathcal P_{k+1}$, a contradiction.
If $Z_k(\bar\alpha)<\infty$, the same conclusion follows by observing that in this case $L=0$ and $Z_k(\alpha_i)$ is bounded below by $\bar r/2$, where $\bar r$ the first value of $r>0$ where $u(\cdot,\bar\alpha)$ takes the value $\beta^+$.
\end{proof}

\begin{remark}\label{rem3}
{\rm Since $\bar\alpha\in\mathcal G_k\subset\mathcal N_{k-1}$, by Proposition \ref{open-sets}, there exists $\delta_0>0$ such that
 $$(\bar\alpha-\delta_0,\bar\alpha+\delta_0)\subset\mathcal N_{k-1}=\mathcal P_k\cup\mathcal G_k\cup\mathcal N_k,$$
therefore, from the previous lemma, there exists $\delta\in(0,\delta_0)$ such that}
$$(\bar\alpha-\delta,\bar\alpha+\delta)\subset \mathcal P_k\cup \mathcal G_k\cup \mathcal P_{k+1}.$$
\end{remark}

\begin{proof}[{\bf Proof of Theorem \ref{main}}]
 From Lemma \ref{nn2}, for each $k\in\mathbb N$ the set $\{\alpha\ge\beta^+\ |\ (\alpha,\gamma^+)\subseteq \mathcal N_k\}$ is nonempty and thus we can set
$$\alpha_k^*:=\inf\{\alpha\ge\beta^+\ |\ (\alpha,\gamma^+)\subseteq \mathcal N_k\}.$$
Clearly,
$$(\alpha_k^*,\gamma^+)\subseteq \mathcal N_k.$$
We will prove that $\alpha_k^*\in\mathcal G_k$ for all $k$. The result will follow by induction over $k$.

 $\alpha_1^*\in \mathcal G_1$: Indeed, if not, then by definition $\alpha_1^*\in{\mathcal N}_1\cup{\mathcal P}_1$, which cannot be because from Proposition \ref{open-sets}, the sets ${\mathcal N}_1$ and ${\mathcal P}_1$ are open and disjoint and $\beta^+\in\mathcal P_1$.

 Assume that the assertion is true for $k=j$, that is, $\alpha_j^*\in\mathcal G_j$. As before, by Proposition \ref{open-sets}, $\alpha_{j+1}^*$ cannot belong to $\mathcal P_{j+1}\cup\mathcal N_{j+1}$, so in order to prove the result we only need to prove that  $\alpha_{j+1}^*\in\mathcal N_j$.
 From Lemma \ref{p2m}, there exists $\bar\alpha\in(\alpha_j^*,\gamma^+)$ such that $\bar\alpha\in\mathcal P_{j+1}$. As $(\alpha^*_{j+1},\gamma^+)\in\mathcal N_{j+1}$, it must be that $\bar\alpha\le \alpha^*_{j+1}$, hence $\alpha_j^*<\alpha_{j+1}^*$ and thus $\alpha_{j+1}^*\in\mathcal N_j$.

\end{proof}

\section{ Concluding remarks and examples}\label{final}

We begin this section by discussing Remark \ref{rem1}.
Assume  that condition $(SC)$ holds  and $s_0$ is such that $sf(s)\ge 0$ $F(s)\ge 0$ for $|s|\ge s_0$. Then there exists $\varepsilon>0$, and $\bar s_0\ge s_0$ such that for $s\ge \bar s_0$,
\be\label{car1}
(i)\quad sf(s)\le (m^*-\varepsilon)F(s)\quad\mbox{and hence}\quad (ii)\quad\frac{F(s)}{s^{m^*-\varepsilon}}\quad \mbox{decreases.}
\ee
Let $\theta\in(0,1)$ be fixed and let $s\ge \bar s_0/\theta$.
Then there exists a positive constant $C$  so that for any $s_2\in[\theta s,s]$ we have
$$Q(s_2)=(Nm-(N-m)\frac{s_2f(s_2)}{F(s_2)})F(s_2)\ge C F( s_2)\ge C F(\theta s).$$
Also, for any $s_1\in[\theta s,s]$,
$$\frac{s^{m-1}}{f(s_1)}\ge \frac{s_1^{m-1}}{f(s_1)},$$
and thus
\ben
Q(s_2)(\frac{s^{m-1}}{f(s_1)})^{N/m}&\ge& C F(\theta s)(\frac{s_1^{m-1}}{f(s_1)})^{N/m}=C F(\theta s)(\frac{s_1^{m}}{s_1f(s_1)})^{N/m}\\
\mbox{by \eqref{car1} (i)}&\ge& \bar CF(\theta s)(\frac{s_1^{m}}{F(s_1)})^{N/m}\\
&=& \bar CF(\theta s)(\frac{s_1^{m^*-\varepsilon}}{F(s_1)})^{N/m}s_1^{N-(m^*-\varepsilon)N/m}\\
\mbox{by \eqref{car1} (ii)} &\ge& \bar CF(\theta s)(\frac{(\theta s)^{m^*-\varepsilon}}{F(\theta s)})^{N/m}s^{N-(m^*-\varepsilon)N/m}\\
&=&\bar C(\frac{(\theta s)^{m^*-\varepsilon}}{F(\theta s)})^{(N-m)/m}(\theta s)^{m^*-\varepsilon}s^{N-(m^*-\varepsilon)N/m}\\
&=&\bar C\theta^{m^*-\varepsilon}(\frac{(\theta s)^{m^*-\varepsilon}}{F(\theta s)})^{(N-m)/m}s^{\varepsilon(N-m)/m}\\
\mbox{by \eqref{car1} (ii)}&\ge& C\theta^{m^*-\varepsilon}(\frac{(\bar s_0)^{m^*-\varepsilon}}{F(\bar s_0)})^{(N-m)/m}s^{\varepsilon(N-m)/m}.
\een
Since the argument is the same for $-s$ large, $(f_3)$ is satisfied.
\medskip

We note that the canonical example $f(s)=|s|^{p-2}s-|s|^{q-2}s$ (which is not Lipschitz at $0$ for $q<2$) satisfies condition $(SC)$ for $1<q<p<m^*$ and thus our theorem applies to this $f$.
\medskip

Next we  deal with \cite[Theorem 2]{fg}, that is we consider $f$ satisfying $(f_1)$ and $(f_2)$ such that for some $s_0>\max\{2e,2\beta^+\}$ and $\lambda>\frac{m}{N-m}$, it holds that
$$f(s)=\frac{|s|^{m^*-2}s}{(\log(|s|))^\lambda}, \quad |s|\ge s_0,$$
where as usual $m^*=Nm/(N-m)$. It can be easily verified that in this case
$$\limsup_{|s|\to\infty}\frac{sf(s)}{F(s)}=m^*
$$
and thus $(SC)$ is not satisfied. We will see that this $f$ satisfies $(f_3)$ and thus  problem \eqref{eq2} with this $f$ has bound states having any prescribed number of zeros.

We first observe that for $s\ge s_0$,
$Q$ is an increasing function. Indeed,
$$Q'(s)=(N(m-1)+m)f(s)-(N-m)sf'(s)=\lambda(N-m)f(s)(\log(s))^{-1}>0,$$
hence for $s_2\in(\theta s,s)$, $Q(s_2)\ge Q(\theta s)$. On the other hand, for $s_1\in[\theta s,s]$,
\ben
\Bigl(\frac{s^{m-1}}{f(s_1)}\Bigr)^{N/m}&=&\Bigl(s^{m-1}s_1^{1-m^*}(\log(s_1))^\lambda \Bigr)^{N/m}\\
&\ge &\Bigl(s^{m-m^*}(\log(\theta s))^\lambda \Bigr)^{N/m}\\
&=&\Bigl((\theta s)^{m-m^*}(\log(\theta s))^\lambda \Bigr)^{N/m}\theta^{(m^*-m)\frac{N}{m}}
\een
and we conclude that
$$\inf_{s_1,s_2\in[\theta s,s]}Q(s_2)\Bigl(\frac{s^{m-1}}{f(s_1)}\Bigr)^{N/m}\ge \theta^{(m^*-m)\frac{N}{m}}Q(\theta s)\Bigl(\frac{(\theta s)^{m-1}}{f(\theta s)}\Bigr)^{N/m}.$$
Hence, it suffices to show that
$$\lim_{s\to\infty}Q(s)\Bigl(\frac{s^{m-1}}{f(s)}\Bigr)^{N/m}=\infty.$$

From the definition of $f$, it holds that there exist a constant $C_0>0$ such that for $s>s_0$,
$$F(s)((\log(s))^\lambda=C_0+\frac{s^{m^*}}{m^*}+\lambda\int_{s_0}^s\frac{F(t)}{t}(\log(t))^{\lambda-1}dt,$$
hence
\ben
NmF(s)-(N-m)sf(s)&=&(\log(s))^{-\lambda}\Bigl(NmF(s)(\log(s))^{\lambda}-(N-m)s^{m^*}\Bigr)\\
&=&(\log(s))^{-\lambda}\Bigl(C_0Nm+Nm\lambda\int_{s_0}^s\frac{F(t)}{t}(\log(t))^{\lambda-1}dt\Bigr)
\een
and thus
$$Q(s)\Bigl(\frac{s^{m-1}}{f(s)}\Bigr)^{N/m}=\frac{C_0Nm+Nm\lambda\int_{s_0}^s\frac{F(t)}{t}(\log(t))^{\lambda-1}dt}{s^{m^*}(\log(s))^{\lambda(m-N)/m}}.$$
Since the denominator in this expression tends to infinity as $s\to\infty$, we may apply L'Hospital's rule to obtain that
\ben
\lim_{s\to\infty}Q(s)\Bigl(\frac{s^{m-1}}{f(s)}\Bigr)^{N/m}
\!\!\!\!&=&\!\!\!\!\lim_{s\to\infty}\frac{\lambda NmF(s)}{s^{m^*}(m^*(\log(s))^{1-\frac{\lambda N}{m}}+\lambda(\frac{m-N}{m})(\log(s))^{-\frac{\lambda N}{m}})}\\
\!\!\!\!&=&\!\!\!\!\lim_{s\to\infty}\frac{\lambda NmF(s)}{s^{m^*}(\log(s))^{-\lambda N/m}(m^*\log(s)+\lambda(\frac{m-N}{m}))}\\
\!\!\!\!&=&\!\!\!\!\lim_{s\to\infty}\frac{\lambda Nmf(s)}{(s^{m^*}(\log(s))^{-\frac{\lambda N}{m}})'(m^*\log(s)+\lambda(\frac{m-N}{m}))+m^*s^{m^*-1}(\log(s))^{-\frac{\lambda N}{m}}}\\
\!\!\!\!&=&\!\!\!\!\lim_{s\to\infty}\frac{\lambda Nm(\log(s))^{1-\lambda+\frac{\lambda N}{m}}}
{(m^*\log(s)-\lambda N/m)(m^*\log(s)+\lambda(\frac{m-N}{m}))+m^*\log(s)}
\een
As $1-\lambda+\frac{\lambda N}{m}>2$ by assumption $\lambda>m/(N-m)$, the result follows and we may apply Theorem \ref{main} to obtain that  problem $(P)$ has solutions with any arbitrary number of nodes.
\medskip

We end this article by giving an example of a nonlinearity for which $\gamma^+$ and $-\gamma^-$ are finite, and an example  for which $\gamma^+<\infty$ and $\gamma^-=-\infty$.

Let $g_1:\mathbb R\to\mathbb R$ be the odd extension of
$$g_1(x)=\begin{cases}
x^3-\sqrt{2x}\quad\mbox{if } x\in[0,2],\\
8-x\quad\mbox{if } x\in[2,8],\\
h(x)\quad\mbox{if } x\in[8,\infty),
\end{cases}
$$
and let $g_2:\mathbb R\to\mathbb R$ be defined by
$$g_2(x)=\begin{cases}
6|x+1|^{-1/3}(x+1)^{-1}\quad\mbox{if } x\in(-\infty,-2),\\
x^3-x\quad\mbox{if } x\in[-2,2],\\
8-x\quad\mbox{if } x\in[2,8],\\
h(x)\quad\mbox{if } x\in[8,\infty),
\end{cases}
$$
where  $h$ is {\em any continuous} function such that $h(8)=0$. Note that $g_1$ satisfies  assumptions  $(f_1)$, $(f_2)$ and $(f_4)$, with $\gamma^-=-8$, $\gamma^+=8$, $\beta^-=-\Bigl(\frac{128}{9}\Bigr)^{1/5}$, $\beta^+=\Bigl(\frac{128}{9}\Bigr)^{1/5}$, and $g_2$ satisfies  assumptions  $(f_1)$, $(f_2)$, $(f_3)$ and $(f_4)$ with $\gamma^+=8$, $\beta^+=\sqrt{2}$, $\beta^-=-\sqrt{2}$ and $\gamma^-=-\infty$.
Note that {\em no restriction} on the growth of $h$ at infinity is needed.

By Theorem \ref{main}, we conclude that for $i=1, 2$, the problem
$$\Delta u+g_i(u)=0,\quad x\in\mathbb R^N,\quad \lim_{|x|\to\infty}u(x)=0,$$
has radially symmetric solutions having any prescribed number of nodes.

\end{document}